\documentclass[letter,english]{article}

\usepackage[utf8]{inputenc}%caractères accentués
\usepackage[T1]{fontenc}
\usepackage{babel}
\usepackage[babel]{csquotes}
\usepackage{xcolor}

\usepackage{authblk}

\usepackage{mathptmx}
\usepackage{graphicx}
\usepackage{subfig}
\usepackage{url}

\usepackage{enumerate}

\usepackage{amsmath}%substack, displaystyle
\usepackage{amsfonts}%mathbb, mathfrak, mathcal
\usepackage{amsthm}%newtheorem, proof
\usepackage{bbm}%fonction indicatrice
\usepackage{amssymb}%sphericalangle
\usepackage{url}%tilde dans les url

\usepackage{float}%choisir l'emplacement des images
\usepackage{tikz}
\usetikzlibrary{shapes.misc}
\tikzset{cross/.style={cross out, draw, 
         minimum size=2*(#1-\pgflinewidth), 
         inner sep=0pt, outer sep=0pt}}
\usepackage{IEEEtrantools}%equations trop longues

\usepackage[pdftex,colorlinks]{hyperref}

\newtheorem{theorem}{Theorem}
\newtheorem{proposition}{Proposition}
\newtheorem{lemma}{Lemma}
\newtheorem{corollary}{Corollary}
\theoremstyle{definition}

\newtheorem{example}{Example}

\DeclareMathOperator{\id}{id}

\DeclareMathOperator{\diag}{diag}

\newcommand{\trace}{\mathrm{Tr}}
\newcommand{\GL}{\mathrm{GL}}
\newcommand{\SL}{\mathrm{SL}}

\newcommand{\C}{\mathbb{C}}
\newcommand{\R}{\mathbb{R}}

\newcommand{\Q}{\mathbb{Q}}

\newcommand{\N}{\mathbb{N}}

\newcommand{\eps}{\varepsilon}

\title{On  the joint spectral radius}
\author{Emmanuel Breuillard \thanks{DPMMS, Wilberforce Road, University of Cambridge, CB30WB, U.K., breuillard@maths.cam.ac.uk}}

\begin{document}

\maketitle

\par\vspace*{.01\textheight}{\centering \emph{Dedicated to the memory of Jean Bourgain} \par}

\begin{abstract}
For a bounded subset $S$ of $d\times d$ complex matrices, the Berger-Wang theorem and Bochi's inequality allow to approximate the joint spectral radius of $S$ from below by the spectral radius of a short product of elements from $S$. Our goal is two-fold: we review these results, providing self-contained proofs, and we derive an improved version with explicit bounds that are polynomial in $d$. We also discuss other complete valued fields.
\end{abstract}

\section{Introduction}

\bigskip

We denote by $\|\cdot\|$ a norm on $\C^d$ and its associated operator norm on the ring of $d\times d$ matrices $M_d(\C)$. For a bounded subset $S \subset M_d(\C)$ we let $\|S\|:=\sup_{s \in S} \|s\|$. The \emph{joint spectral radius} \cite{rota-strang, berger-wang, daubechies-lagarias, wirth, jungers} is defined by:
\begin{equation}\label{joint} \rho(S):=\lim_{n \to +\infty} \|S^n\|^{\frac{1}{n}}\end{equation} where $S^n:=\{s_1\cdot\ldots\cdot s_n, s_i \in S\}$ is the $n$-th fold product set.
From the submultiplicativity of the operator norm it is clear that the limit exists, is independent of the choice of norm and coincides with the infimum of all $\|S^n\|^{\frac{1}{n}}$, $n \ge1$. A straightforward consequence is that $S \mapsto \rho(S)$ is upper-semicontinuous for the Hausdorff topology. Moreover $\rho(S^k)=\rho(S)^k$ for every $k\in \N$. It is also clear that $\rho(gSg^{-1})=\rho(S)$ for every $g \in \GL_d(\C)$. Rota and Strang \cite{rota-strang} observed that $\rho(S)$ is equal  to the infimum of $\|S\|$ as the norm varies among all possible norms on $\C^d$. %This infimum is actually a minimum \cite{barabanov, wirth} if $S$ is irreducible (i.e. does not preserve a proper subspace of $\C^d$). 
Combined with John's ellipsoid theorem this easily yields: 

\begin{lemma}\label{john}  Given a norm $\|\cdot\|$ on $\C^d$, for any bounded subset $S \subset M_d(\C)$, we have:
$$\rho(S) \leq \inf_{g \in \GL_d(\C)} \|gSg^{-1}\| \leq d \cdot \rho(S).$$
\end{lemma}

When $S$ is irreducible (i.e. does not preserve a proper subspace of $\C^d$) it turns out that there is norm such that $\rho(S)=\|S\|$. The existence of such \emph{extremal norms} will be reviewed in Section \ref{barabanov} along with related known facts. It also follows easily from this that $\rho(S)=0$ if and only if the subalgebra $\C[S]$ generated by $S$ is nilpotent.

It turns out that $\rho(S)$ can also be approximated from below by eigenvalues. Let $\Lambda(s)$ be the largest modulus of an eigenvalue of $s \in M_d(\C)$ and $$\Lambda (S):=\max_{s \in S} \Lambda(s).$$  It is clear that $\Lambda(S) \leq \rho(S)$ and thus $\Lambda(S^n)^{\frac{1}{n}} \leq \rho(S)$ for all $n$. When $S$ is a singleton, the classical Gelfand-formula asserts that $\Lambda(s)=\rho(\{s\})$. For several matrices the key fact is as follows:

\begin{theorem}[Berger-Wang \cite{berger-wang}] \label{bg}$$\rho(S)=\limsup_{n \to +\infty} \Lambda(S^n)^{\frac{1}{n}}.$$
\end{theorem}

An immediate consequence is that $S \mapsto \rho(S)$ is also lower-semicontinuous and hence continuous for the Hausdorff topology. Theorem \ref{bg}  had been conjectured by Daubechies and Lagarias \cite{daubechies-lagarias}. Elsner  \cite{elsner} gave a simple proof of it. In this article we will be interested in giving explicit estimates quantitying this convergence. Our first observation is that in fact the following slightly stronger result holds:

\begin{theorem}\label{el} Let $S\subset M_d(\C)$ be a bounded subset with $\rho(S)>0$. Then $$\limsup_{n \to +\infty} \frac{\Lambda(S^n)}{\rho(S)^n}=1.$$
\end{theorem}

The question of the speed of convergence in Theorem \ref{bg} or \ref{el} is an interesting one and goes back at least to the Lagarias-Wang finiteness conjecture \cite{lagarias-wang}, which posited that the limsup should be attained at a certain finite $n$. This has been disproved by Bousch and Mairesse \cite{bousch-mairesse} for $2 \times 2$ matrices (see also \cite{jenkinson-pollicott,jems, hare-sidorov-morris}) and Morris (see \cite[Thm 2.7]{oregon-reyes-gromov}) gave an example with  $S=\{a,b\}\subset \SL_2(\R)$. In general counter-examples are thought to be rare. %In fact (e.g. \cite{breuillard-sert}) if $a,b\in SL_2(\R)$ correspond to hyperbolic isometries of the hyperbolic plane with disjoint axes in the same direction and large enough translation length, then the set of $t \in [1,+\infty)$ such that $$S:=\{a,b^t\}$$
%violates the finiteness conjecture is uncountable and of Lebesgue measure zero. %In these examples however the speed of convergence exponential convergence rate.

Elsner's proof of Theorem \ref{bg} is based on a pigeonhole argument, which we will revisit in this note and can roughly be described as follows: if $\rho(S)=1$, then given a unit vector $x \in \C^d$ we may always find $s \in S$ such that $sx$ is also a unit vector (this follows from the existence of a Barabanov norm, see Section \ref{barabanov}), so iterating this construction we eventually find a short product $w=s_n\cdot \ldots \cdot s_k$ with $wy$ close to $y=s_{k-1}\cdot \ldots \cdot s_1 x$, implying that $w$ has an eigenvalue close to $1$. This idea also leads to a proof of Theorem \ref{el} and to the following quantitative and explicit version:

\begin{theorem}\label{bg-el}  Let $S\subset M_d(\C)$ be a bounded subset with $\rho(S)=1$. Set $n_0(d)=3^d4^{d^2}$ and let $\eps>0$. If $n \ge \eps^{-d^2}n_0(d)$, then
$$\max_{k \leq n} \Lambda(S^k) \ge 1-\eps. $$
\end{theorem}

\iffalse

\begin{theorem}\label{oregon-el}  Let $A_1,\ldots,A_n$ be matrices in $M_d(\C)$ with operator norm at most $1$ (for some norm on $\C^d$). Let $\eps>0$. If $n \ge \eps^{-d^2}n_0(d)$, then
$$\max_{1\leq \alpha \leq \beta \leq n} \rho(A_\alpha \cdot \ldots \cdot A_\beta) \ge (1-\eps) \|A_1\cdot \ldots \cdot A_n\|. $$
\end{theorem}

This improves the main inequality proved by Oregon-Reyes in \cite{oregon-reyes}, which we will discuss further below.
\fi

This yields a polynomial decay of the form $|\sup_{k\leq n} \Lambda(S^k)-1|=O_{S,d}(n^{-1/d^2})$ in Theorem \ref{el} when $\rho(S)=1$. In \cite{morris-adv} Morris proved a much stronger super-polynomial upper bound on the speed of convergence: that is $|\sup_{k\leq n} \Lambda(S^k)-1|=O_{A,S}(n^{-A})$ for every $A\ge 1$, provided $S$ is finite and $\rho(S)=1$. However the implied constant is not explicit. He also points out that his argument fails when $S$ is infinite. %Nevertheless his result implies Theorem \ref{el} also for infinite $S$ simply by approximating $S$ by finite subsets in view of the continuity of the joint spectral radius.

 In this note we will be interested in the $d$ aspect. The bound on $n$ in Theorem \ref{bg-el} is super-exponential in $d$. If we aim to approximate the joint spectral radius no longer up to a small error, but only up to a constant multiple, we can expect  polynomial bounds in $d$. In this vein Bochi \cite[Theorem  B]{bochi} established the following general inequality:

\begin{theorem}[Bochi \cite{bochi}]\label{boc} There are constants $c(d)>0,N(d)>0$ such that for every bounded set $S \subset M_d(\C)$ we have:
\begin{equation}\label{bocc}\max_{1 \leq k \leq N(d)} \Lambda(S^k)^{\frac{1}{k}} \ge c(d) \cdot \rho(S).\end{equation}
\end{theorem}

Note that Theorem \ref{bg} (but not Theorem \ref{el}) follows immediately from Bochi's inequality: indeed apply the inequality to $S^n$ and let $n$ tends to infinity. On the other hand Theorem \ref{bg-el} implies Bochi's inequality with $N(d)=3^d8^{d^2}$ and $c(d)=\frac{1}{2}$ say. We are interested in quantifying the constants $c(d)$ and $N(d)$ in terms of the dimension $d$. Example \ref{exd} (2) below shows that $N(d)\ge d$. Bochi's proof gave $N(d)=2^d-1$, but a non-constructive $c(d)$ obtained via a topological argument involving some geometric invariant theory. 

In \cite[2.7, 2.9]{breuillard-height} another non-constructive proof was given with $N(d)=d^2$. This proof actually allows to take for $N(d)=\ell(d)$ the smallest upper bound on the integer $k$ such that for any $S\subset M_d(\C)$ the powers $S,\ldots,S^{k}$ span linearly the matrix algebra $\C[S]$ generated by $S$. It is immediate that $\ell(d)\leq d^2$, but in a recent breakthrough Shitov \cite{shitov} has proved that $\ell(d)\leq 2d(\log_2 d +2)$ greatly improving an earlier bound in $O(d^{3/2})$ due to Pappacena \cite{pappacena}.

In order to motivate our main result and since it is very short, we give now a direct proof of Theorem \ref{boc} using the following slight variant of the argument from \cite{breuillard-height}:  after rescaling to $\rho(S)=1$ and setting $c(d)=c'(d)/d$ and $N(d)=\ell(d)$, Claim 1 immediately implies Theorem \ref{boc}.\\

\noindent {\bf Claim 1.} {\it There is $c'(d)>0$ such that for every bounded subset $S$ of $M_d(\C)$ with $\rho(S)=1$ there is a non-zero idempotent $p \in M_d(\C)$ (i.e. $p^2=p$) such that $c'(d) p$ belongs to the complex convex hull of $S,\ldots,S^{\ell(d)}$.}

\begin{proof}By the complex convex hull $Conv(Q)$ of $Q \subset M_d(\C)$, we mean the set of linear combinations $\alpha_1q_1+\ldots+\alpha_nq_n$ with $q_i \in Q$ and $|\alpha_1|+\ldots+|\alpha_n|=1$.  Since the problem is invariant under conjugation, in view of Lemma \ref{john} we may assume that $S$ is confined to a bounded region of $M_d(\C)$, allowing us to pass to a Hausdorff limit of potential counter-examples to the claim. By compactness and upper semi-continuity of the joint spectral radius we get a bounded subset $S$ with $\rho(S)\ge1$, but such that $Conv(S \cup \ldots \cup S^{\ell(d)})$ contains no scalar multiple of  an idempotent. In particular $\C[S]$ contains no  idempotent. By the Artin-Wedderburn theorem this means that $\C[S]$ is a nilpotent subalgebra of $M_d(\C)$. In particular $S^d=0$, which is in contradiction with $\rho(S^d)=\rho(S)^d\ge1$.
\end{proof}

As with Bochi's original argument, this one does not give any explicit estimate on the constant $c(d)$. It is however possible to ``effectivise'' the argument just given: this requires effectivising the proof of Wedderburn's  theorem and, after a fairly painstaking analysis, the details of which we will spare the reader, yields a rather  poor lower bound on $c(d)$ of doubly exponential type in $d$.  Another route is to use an idea appearing in the work of Oregon-Reyes \cite[Rk. 4.5]{oregon-reyes}, which consists in using the effective arithmetic nullstellensatz by making explicit the implication $\{\trace(S^k)=0$ for all $k=1,\ldots,\ell(d)\}$ $\Rightarrow$ $\{S^d=0\}$. This also yields an effective bound on $c(d)$, which is again unfortunately  rather poor, at least doubly exponential in $d$. 

%Recently Oregon-Reyes \cite[Thm 1.2]{oregon-reyes} obtained a strengthening of Bochi's inequality, which we will discuss further below. In doing so he provided explicit values for $c(d)$ (with $N(d)=2^d$) by means of the effective arithmetic nullstellensatz \cite[Rk. 4.5]{oregon-reyes}. The same idea can be applied to the above proof. The implication $\{\trace(S^k)=0$ for all $k=1,\ldots,\ell(d)\}$ $\Rightarrow$ $\{S^d=0\}$ means that the polynomials describing the vanishing of $S^d$ belong to the radical ideal generated by those expressing the vanishing of the traces $\trace(S^k)$, so applying the effective nullstellensatz does yield an explicit bound on $c(d)$ (with $N(d)=\ell(d)$). Unfortunately it is again rather poor,  at least doubly exponential in $d$. 

The following result, which is the main contribution of this note, gives explicit polynomial bounds on both $c(d)$ and $N(d)$. %As with the effective nullstellensatz it also relies on  a certain pigeonhole argument. 

\begin{theorem}\label{pol-bd} For every bounded set $S \subset M_d(\C)$ we have:
$$\max_{1 \leq k \leq 2d^3} \Lambda(S^k)^{\frac{1}{k}} \ge \frac{1}{2^8d^5} \cdot \rho(S).$$
In particular
$$\max_{1 \leq k \leq N_2(d)} \Lambda(S^k)^{\frac{1}{k}} \ge \frac{1}{2} \cdot \rho(S),$$
where $N_2(d)=2d^3\lceil 8+5\log_2 d\rceil$.
\end{theorem}

The factor $\frac{1}{2}$ can of course be replaced by any number $\kappa<1$ provided $N_2(d)$ is replaced by $N_{\kappa^{-1}}(d):=2d^3\lceil\log_{\kappa^{-1}} (2^8d^5)\rceil$. The proof exploits a different kind of pigeonhole argument, where one argues, as in the classical Siegel lemma in number theory, that some non-zero linear combination with small integer coefficients of the iterates $s_n\cdot \ldots \cdot s_1x$ will vanish or be very small. In turn, this forces one of the products to have a spectral radius bounded away from zero. 

The following natural questions then arise: \bigskip

\noindent \emph{Questions:} How sharp is the bound $d^{3+o(1)}$ on $N_2(d)$? We only know that $N_2(d)$ must be at least $d$. Is there a polynomial bound on $c'(d)$ in Claim 1 above?\\

In \cite[Theorem A]{bochi} Bochi proves another inequality, giving this time a lower bound on $\rho(S)$ in terms of the norms of $S^n$, which, when iterated, gives a speed of convergence for $(\ref{joint})$, see \cite{kozboc}.  Given any norm $\|\cdot\|$ on $\C^d$, 

\begin{equation}\label{bocA} \|S^d\| \leq C_0(d) \rho(S) \|S\|^{d-1}.
\end{equation}

While no explicit bound  on $C_0(d)$ was given in \cite{bochi}, his proof gives a super-polynomial bound in $d^{3d/2}$ (see \cite[Section 4]{kozboc}). It turns out that the pigeonhole argument for our Theorem \ref{pol-bd} gives a polynomial bound for $(\ref{bocA})$ at the expense of increasing the power of $S$:

\begin{theorem}\label{bocnewA} Let $S\subset M_d(\C)$ be a bounded subset and set $n_1=2d^2$. Then 
\begin{equation}\label{bocAplus}\|S^{n_1}\| \leq 2^7 d^4 \rho(S) \|S\|^{n_1-1}.\end{equation}
\end{theorem}

Iterating $(\ref{bocAplus})$ yields an explicit estimate quantifying the convergence in $(\ref{joint})$ improving the bounds obtained in \cite[Theorem 1]{kozboc}.

%We now move on to a related inequality, recently obtained by Oregon-Reyes in \cite{oregon-reyes}.

%The second assertion of the theorem follows immediately from the first by replacing $S$ by $S^n$ for suitable $n$. 

%Our proof of Theorem \ref{main1} actually works in two regimes: in the first we are trying to find a short word that is non-nilpotent, and this yields the bounds of Theorem \ref{main1}, while in the second we are trying to find a short word with an eigvenvalue close to the unit circle. This yields the following:

%\begin{theorem} For every bounded set $S \subset M_d(\C)$ and every $\eps>0$ we have:
%$$\max_{1 \leq k \leq N_3(d)\eps^{-1}} \frac{\Lambda(S^k)}{\rho(S)^k} \ge 1-\eps.$$
%In particular
%\end{theorem}

Finally we examine what happens when the field $\C$ is replaced by an arbitrary algebraically closed complete valued field $(K,|\cdot|)$. By Ostrowski's theorem, if $K$ is not $\C$ it must be non-archimedean (for instance $\C_p$ the completion of the algebraic closure of the field of $p$-adic numbers $\Q_p$, or the completion of the field of Laurent series over the algebraic closure of $\mathbb{F}_p$). All of the above makes sense of course and the joint spectral radius is defined in the same way. As it turns out, the analogues of the results above are much simpler for such $K$, the Lagarias-Wang finiteness conjecture holds in a uniform way, and in fact:

\begin{theorem}\label{ultra-thm} Let $K$ be an algebraically closed  non-archimedean complete valued field. Consider an ultrametric norm $\|\cdot \|_0$ on $K^d$ and a bounded subset $S$ of $M_d(K)$. Then
\begin{equation}\label{boc-ultra}\max_{1 \leq k \leq \ell(d)} \Lambda(S^k)^{\frac{1}{k}} = \rho(S)=\inf_{g \in \GL_d(K)} \|gSg^{-1}\|_0.\end{equation}
Moreover, $\rho(S)>0$ if and only if the subalgebra generated by $S$ is not nilpotent, in which case there is an ultrametric norm $\|\cdot\|$ on $K^d$ with $\|S\|=\rho(S)$.
\end{theorem}

%\begin{lemma}\label{john-ultra} Given an ultrametric norm $\|\cdot\|$ on $K^d$, for any bounded subset $S \subset M_d(K)$ we have:
%$$\rho(S)=\inf_{g \in \GL_d(K)} \|gSg^{-1}\|.$$
%\end{lemma}

Recall that $\ell(d)$ denotes the smallest integer $k$ such that for any field $F$ and any $S \subset M_d(F)$ the power sets $S,\ldots,S^k$ span linearly the algebra $F[S]$. Obviously $\ell(d)\leq d^2$ and recall that in fact $\ell(d)\leq 2d\log_2d+4d-4$ by \cite{shitov}.

 If $K$ is not algebraically closed, then $(\ref{boc-ultra})$ still holds with $K$ replaced by its algebraic closure $\overline{K}$ (indeed, the absolute value extends uniquely to $\overline{K}$ and the completion of $\overline{K}$ will remain algebraically closed by K\"ursch\'ak's theorem \cite[5.J.]{ribenboim}). 

In the special case when $K$ is a local field and $S$ a compact subgroup of $\GL_d(K)$, the last assertion of the theorem recovers the well-known Bruhat-Tits fixed point theorem: the norm $\|\cdot\|$ will be preserved by $S$ and thus be a fixed point in the Bruhat-Tits building of ultrametric norms \cite{goldman-iwahori}.

Theorem \ref{ultra-thm} was proved in \cite{breuillard-height} for $K=\C_p$. We will give a slightly more direct  proof of the general case. 

Similarly, the analogue of Theorem \ref{bocnewA} reads:

\begin{theorem}\label{bocA-ultra} For any ultrametric norm $\|\cdot\|_0$ on $K^d$ and $S \subset M_d(K)$ bounded
\begin{equation}\|S^d\|_0\leq \rho(S) \|S\|_0^{d-1}.\end{equation}
\end{theorem}

%\begin{theorem} Let $K$ be a complete non-archimedean valued field and $S$ be a bounded subset of $M_d(K)$. Then 
%\begin{equation}\label{boc-ultra}\max_{1 \leq k \leq \ell(d)} \Lambda(S^k)^{\frac{1}{k}} = \rho(S).\end{equation}
%\end{theorem} 

\section{Extremal norms and Barabanov norms}\label{barabanov}

In this section we recall some well-known facts about the joint spectral radius and extremal norms providing complete and self-contained proofs. Most of the material can be found in the first chapters of the book \cite{jungers}. We then prove Theorem \ref{el}.

We begin by the observation of Rota and Strang mentioned in the introduction.  Recall that $\|\cdot\|$ denotes both a norm on $\C^d$ and its associated operator norm and that for some subset $Q$ (in either $\C^d$, or $M_d(\C)$), we set  $\|Q\|:=\sup_{q \in Q} \|q\|$.

\begin{lemma}[Rota-Strang] \label{rotas}Let $S\subset M_d(\C)$ be a bounded subset. \begin{equation}\label{rs}\rho(S)=\inf_{\|\cdot\|} \|S\|,\end{equation} where the infimum is over all norms  on $\C^d$.
\end{lemma}

\begin{proof} Let $r>0$ with $r\rho(S)<1$ and consider the norm $v_r(x):=\sum_{n \ge 0} \|S^nx\|r^n$. Clearly $v_r(sx)\leq \frac{1}{r}v_r(x)$ for all $s \in S$. So $v_r(S)\leq r^{-1}$. Letting $r^{-1}$ tend to $\rho(S)$ yields the result.
\end{proof}

Lemma \ref{john} follows immediately by combining Lemma \ref{rotas} with the following well-known fact:

\begin{lemma}[John's ellipsoid] If $v$ is a norm on $\C^d$ and $\|\cdot\|_2$ the standard hermitian norm, then there is $g \in \GL_d(\C)$ such that for all $x \in \C^d$ $$\|gx\|_2 \leq v(x) \leq \sqrt{d} \cdot \|gx\|_2.$$ In particular if $w(x)$ is any another norm, then for some $h \in \GL_d(\C)$ $$w(hx) \leq v(x) \leq d \cdot w(hx).$$
\end{lemma}

\begin{proof} According to John's ellipsoid theorem (e.g. \cite{ball}) every symmetric convex body $K$ in $\R^k$ contains a unique ellipsoid $E$ of maximal volume and $E$ moreover satisfies $K \subset \sqrt{d}E$. If $K$ is the ball of radius $1$ of the complex norm $v$ in $\C^d=\R^{2d}$, then the uniqueness implies that the norm associated to $E$ is hermitian, hence of the form $\|gx\|_2$ for some $g \in \GL_d(\C)$.  
\end{proof}

\noindent \emph{Remark.} This argument shows that the constant $d$ in Lemma \ref{john} can be replaced by $\sqrt{d}$ if the norm is $\|\cdot\|_2$. In fact a more subtle argument (see e.g. \cite{blondel-nesterov}) shows that it can be replaced with $\sqrt{\min\{k,d\}}$ in case  $S$ has $k$ elements. \\

One says that $S$ is \emph{irreducible} if it does not preserve a non-trivial proper subspace of $\C^d$. It is said to be \emph{product bounded} if the semigroup it generates $T:=\bigcup_{n \ge 1} S^n$ is bounded. The following is  also classical (see \cite{barabanov, berger-wang, elsner, wirth}):

\begin{lemma}[Extremal norms] \label{pded} Suppose $S$ is irreducible. Then $\rho(S)>0$, $S/\rho(S)$ is product bounded and the infimum in $(\ref{rs})$ is attained.
\end{lemma}
Norms realising the infimum in $(\ref{rs})$ are called \emph{extremal norms}.

\begin{proof}By Burnside's theorem the subalgebra $\C[S]$ generated by $S$ is all of $M_d(\C)$. Since $\C[S]$ is linearly spanned by $S \cup \ldots \cup S^{d^2}$ we may express each element of the canonical basis $E_{ij}$ of $M_d(\C)$ as a linear combination of elements from $T$. Given that $\trace(E_{ii})=1$, this means that at least one element of $T$ has non-zero trace, which clearly forces $\rho(S)>0$. Rescaling, we may  assume without loss of generality that $\rho(S)=1$. In particular $|\trace(t)|\leq d$ for all $t \in T$ and thus $|\trace(tE_{ij})|$ is bounded independently of $t \in T$, which means that $T$ is bounded. Finally, given any norm $\|\cdot\|$ on $\C^d$ and setting $v(x):=\|Tx\|$, we get a well-defined norm such that $v(sx)\leq v(x)$ for all $s \in S$. Hence $v$ is an extremal norm.
\end{proof}

The example of a single non-trivial unipotent matrix shows that the infimum in $(\ref{rs})$ is not attained in general. If $S$ is not irreducible, it can be put in block triangular form in some basis of $\C^d$. Therefore the following is an immediate consequence of the previous lemma (recall that an algebra $N$ is nilpotent if there is an integer $n$ such that $N^n=0$).

\begin{corollary} Let $S$ be a bounded subset of $M_d(\C)$. Then $\rho(S)=0$ if and only if $\C[S]$ is a nilpotent subalgebra of $M_d(\C)$.
\end{corollary}

If $S$ is irreducible and $\rho(S)=1$, $T$ is bounded and we may define 
\begin{equation}\label{tintynorm} v(x):=\limsup_{n \to +\infty}\|S^nx\|.\end{equation}
Then $v$ is a norm, because $v(x)=0$ for some $x \neq 0$ implies $v(S^nx)=0$ for all $n$, which implies by irreducibility that $v$ is identically zero, and hence that $\rho(S)=0$. In particular:

\begin{lemma}[Barabanov norms] \label{bar} Let $S$ be an irreducible bounded subset of $M_d(\C)$, then there is a complex norm $v$ on $\C^d$ such that for all $x \in \C^d$, \begin{equation}\label{bnorm} \max_{s \in S} v(sx)=\rho(S) \cdot v(x).\end{equation}
\end{lemma}

\begin{proof}Indeed we may define $v$ as in $(\ref{tintynorm})$ for $S$ replaced by $S/\rho(S)$. \end{proof} A norm satisfying $(\ref{bnorm})$ is a special kind of extremal norm called a \emph{Barabanov norm} (see \cite{barabanov, wirth, protasov, kozyakin}). Such norms are not unique in general (e.g. in Example \ref{exd} 4. below any norm $\|\cdot\|$ on $\C^d$ with $\eps \|x\|_2\leq \|x\|\leq \|x\|_2$ is a Barabanov norm for $S$), but they can be in some situations \cite{morris-unique}.

Another object is naturally associated to $S$ when $\rho(S)=1$, it is the \emph{attractor semigroup}   \cite{barabanov, wirth}
$$T_\infty:=\bigcap_{n\ge 1} \overline{S^nT}.$$
In other words this is the set of limit points of finite products $s_1\cdot \ldots \cdot s_n$ whose length $n$ tends to infinity. It is clearly compact and contains elements of norm at least $1$ for any operator norm. Indeed otherwise we would have $\|S^n\|<1$ for some $n$ and thus $\rho(S^n)<1$, which is impossible as $\rho(S^n)=\rho(S)^n=1$. By construction, the Barabanov norm $(\ref{tintynorm})$ is also equal to $v(x)=\max_{t \in T_\infty} \|tx\|$. Furthermore it is straightforward that $T_\infty=T_\infty S=S T_\infty$ and $T_\infty^2=T_\infty$, and that:

\begin{lemma} Suppose $S$ is irreducible with $\rho(S)=1$. Then $T_\infty$ is also irreducible and $\rho(T_\infty)=1$.
\end{lemma}

\begin{proof} For every non-zero $x \in \C^d$ the linear span $\langle T_\infty \rangle x$ contains $\langle T_\infty \rangle S^k x$ for each $k$, and hence $\langle T_\infty \rangle \C^d$ by irreducibility of $S$. So this must be $0$ or $\C^d$. The former is impossible, because $T_\infty  \neq \{0\}$ by the above discussion.  So $T_\infty$ is irreducible. Finally by construction $v(T_\infty)=1$ and $T_\infty^k=T_\infty$ for every $k$. Hence $\rho(T_\infty)=1$.
\end{proof}

We are now in a position to prove Theorem \ref{el}.

\begin{lemma}[Existence of an idempotent]\label{idem} Suppose $S$ is a bounded irreducible subset of $M_d(\C)$ with $\rho(S)=1$. Then the attractor semigroup $T_\infty$ contains a non-zero idempotent.
\end{lemma}

\begin{proof}Let $K$ be the subset of $T_\infty$ made of elements with operator norm $1$. We have already seen that $K$ is non-empty. If $ab$ has norm one and $a,b \in T_\infty$, then both $a$ and $b$ have norm one. So $K \subset K^2$. Starting from some $t_0 \in K$ we may write $t_0=t_1s_1$ with $t_1,s_1 \in K$, and then similarly $t_1=t_2s_2$, etc. For each $n$ we have $t_0=t_ns_n\cdot \ldots \cdot s_1$. By compactness of $K$ there is a subsequence $n_i$ such that $s_{n_i}\cdot \ldots \cdot s_1$ converges, say towards $k \in K$. Passing to a further subsequence we may assume that $s_{n_{i+1}} \cdot \ldots \cdot s_{n_i+1}$ also converges, say towards $u \in K$. At the limit we have $k=uk$.  But there is a unit vector $x$ such that $y:=kx$ has norm $1$. Hence $y=uy$ and $u$ has $1$ as an eigenvalue. So $T_\infty$ contains an element $u$ with eigenvalue $1$. Now looking at $u$ in Jordan normal form and considering large powers of $u$, we see that the Jordan blocks with eigenvalue of modulus $1$ must be of size $1$, because powers of non-trivial unipotents are unbounded. Therefore $\{u^n\}_{n \ge1}$ contains an idempotent in its closure.
\end{proof}

Note that $T_\infty$ may contain $0$, so the lemma does not follow from a general result  guaranteeing the existence of idempotents in compact semigroups such as the Ellis-Numakura lemma.

\begin{proof}[Proof of Theorem \ref{el}] We first assume that $S$ is irreducible. Rescaling, we may assume that $\rho(S)=1$. By Lemma \ref{idem} $T_\infty$ contains an idempotent. In particular $\Lambda(T_\infty)=1$, which implies what we want. The general case follows from the irreducible one. Indeed if $S$ is not irreducible it can be put in block triangular form, and if $S_{ii}$ denotes the $i$-th diagonal block, then it is straightforward to check (either from the definition, or more directly from Theorem \ref{bg}) that $\rho(S)=\max_i \rho(S_{ii})$. 
\end{proof}

\begin{example}\label{exd} The following are examples of irreducible subsets of $M_d(\C)$ with joint spectral radius equal to $1$. \begin{enumerate} \item $S=\{E_{ij}\}_{ij}$ the elementary matrices in $M_d(\C)$. Note that $S$ is made of rank $1$ elements and $T_\infty=S \cup \{0\}$.
\item $S=\{E_{i, i+1}\}_{1\leq i<d} \cup \{E_{d1}\}$. Note that $T=T_\infty=\{0\}\cup\{E_{ij}\}_{ij}$.
\item $S=U_d(\C) \cup \{t\}$, where $U_d(\C)$ is the group of unitary matrices and $t=\diag(\alpha_1,\ldots,\alpha_d)$ with $|\alpha_i|<1$. Then $T_\infty=T \cup \{0\}$. 
\item $S=\{\id\} \cup \eps U_d(\C)$ for $\eps<1$. Then $T_\infty=T\cup \{0\}$.
\end{enumerate}
\end{example}

\section{Explicit bounds for Theorem \ref{el}}
In this section we prove Theorem \ref{bg-el}. We need a basic lemma.

\begin{lemma}\label{reg} Let $\|\cdot\|$ be a norm on $\C^d$. Let $A \in M_d(\C)$ and $x \in \C^d$ with $\|A\|\leq 1$ and $\|x\|=1$. Let $\eps>0$ and $\lambda \in \C$ with $|\lambda|\leq 2$. Assume that $\|Ax-\lambda x\|\leq (\eps |\lambda|)^d$. Then the spectral radius $\Lambda(A)$ of $A$ satisfies $\Lambda(A) \ge |\lambda|(1-4\eps)$.
\end{lemma}

\begin{proof}Writing $A^kx-\lambda^kx = A^{k-1}(Ax-\lambda x) + \ldots + \lambda^{k-1}(Ax-\lambda x)$ and using that $\|A\|\leq 1$ we obtain for $k \le d$ $$\|A^kx-\lambda^kx\|\leq (\eps|\lambda|)^d (1+|\lambda|+\ldots+|\lambda|^{k-1})\leq  (2\eps|\lambda|)^d.$$ If $\chi_A(t)=t^d+a_{d-1}t^{d-1}+\ldots+a_0$ is the characteristic polynomial of $A$, then $\|\chi_A(A)x-\chi_A(\lambda)x\|\leq \sum |a_k|\|A^k x - \lambda^k x\|$, and $\chi_A(A)=0$ by Cayley-Hamilton. But $|a_{d-k}|\leq {d \choose k}$, and the result then follows from $$|\Lambda(A)-|\lambda||^d \leq |\chi_A(\lambda)|\leq 2^d  (2\eps|\lambda|)^d.$$
\end{proof}

\iffalse

A very similar lemma, but in a different regime is as follows:

\begin{lemma}\label{pol-lem}Let $\|\cdot\|$ be a norm on $\C^d$. Let $A \in M_d(\C)$ and $x \in \C^d$ with $\|A\|\leq 1$ and $\|x\|=1$. Assume that $\|Ax- x\|< 1/4d$. Then the spectral radius $\Lambda(A)$ of $A$ satisfies $\Lambda(A) > 1/4d$.
\end{lemma}

\begin{proof}Let $\eps=\|Ax-x\|$ and $\delta=\Lambda(A)\leq \frac{1}{4d}$. Arguing as in the previous proof: $$|\chi_A(1)|\leq \sum_1^d |a_k|\|A^k x -  x\| \leq d\eps \sum_1^d |a_k| \leq d\eps \sum_0^d {d \choose k} \delta^k = d\eps (1+\frac{1}{4d})^d\leq 2d\eps.$$
But using $1-x \leq e^{-x} \leq 1-x/2$ for all $x \in [0,1]$ we have $$|\chi_A(1)| \ge (1-\delta)^d \ge e^{-2d\delta}  \ge 1-2d\delta,$$ a contradiction if $\eps<1/4d$.
\end{proof}

\fi

\begin{proof}[Proof of Theorem \ref{bg-el}] We may put $S$ in block triangular form. Since $\rho(S)= \max_i \rho(S_{ii})$ at least one of the irreducible diagonal blocks $S_{ii}$ has $\rho(S_{ii})=1$. Hence without loss of generality we may assume that $S$ is irreducible. Let $v$ be a Barabanov norm for $S$ as in Lemma \ref{bar}. Pick a unit vector $x \in \C^d$ and find recursively $s_1,s_2,\ldots$ such that $x_n=s_n\cdot \ldots \cdot s_1 x$ satisfies $v(x_n)=1$ for all $n$. Let $\delta=(\eps/4)^d$. Note that the cardinality of a $\delta$-separated set lying in the unit ball for $v$ is at most $(1+\delta/2)^d/(\delta/2)^d=(1+\frac{2}{\delta})^d \leq n_0(d)\eps^{-d^2}$. By pigeonhole, there is $n< n'$ both smaller than this bound such that $v(x_n - x_{n'})<\eps$. Let $A=s_{n'}\cdot \ldots \cdot s_{n+1}$. In other words $v(Ax_n-x_n)<\delta$. By Lemma \ref{reg}, it follows that $\Lambda(A)\ge 1-4\delta^{1/d}.$
\end{proof}

%Pick $x\in \C^d$ with  $\|A_1\cdot \ldots \cdot A_n\|=\|A_1\cdot \ldots \cdot A_n\cdot x\|$ and let $\delta$ be this common value. Let $y_k=A_k\cdot \ldots \cdot A_n\cdot x$. The $y_k$'s belong to the unit ball. Therefore 

\section{Explicit bounds for Bochi's inequalities}
In this section we prove Theorems \ref{pol-bd} and \ref{bocnewA}. We begin by the Siegel-type lemma already mentioned.

\begin{lemma}[Siegel-type lemma]\label{siegel} Let $\|\cdot\|$ be a norm on $\C^d$. Let $\eps \in (0,1)$ and $T,n \in \N$ with $(1+T)^{n} > (1+2nT\eps^{-1})^d$. Pick $x_1,\ldots,x_n$ vectors in $\C^d$ with $\|x_i\|\leq 1$. Then there are integers $c_1,\ldots,c_n$, not all zero, such that $|c_i|\leq T$ for all $i$ and $$\|\sum_1^n c_i x_i\| \leq \eps.$$
\end{lemma}

\begin{proof} Consider the sums $\sum_1^n d_i x_i$ for integers $d_i \in [0,T]$. They have norm at most $Tn$. If all $\frac{\eps}{2}$-balls around them were disjoint, then the ball of radius $Tn+\frac{\eps}{2}$ around the origin would contain at least $(1+T)^n$ disjoint balls of radius $\frac{\eps}{2}$. Comparing volumes we would have $(1+T)^n \leq (Tn+\eps/2)^d/(\eps/2)^d$, contrary to our assumption. Hence two of these balls, corresponding to $(d_i)_i$ and $(d_i')_i$, say, must intersect. Setting $c_i=d_i'-d_i$ we get what we want.
\end{proof}

\begin{lemma}\label{trace}Let $\eps>0$. Let $A \in M_d(\C)$ such that $|\trace(A^k)|\leq \eps^k$ for $k=1,\ldots,d$. Then the spectral radius of $A$ satisfies $\Lambda(A)\leq 2\eps$.
\end{lemma}

\begin{proof} Let $s_k=\lambda_1^k+\ldots+\lambda_d^k$, where $\lambda_1,\ldots,\lambda_d$ are the eigenvalues of $A$. The Newton relations read $s_k+a_{d-1}s_{k-1}+\ldots+a_{d-k+1}s_1=-ka_{d-k}$, where $t^d+a_{d-1}t^{d-1}+\ldots+a_0$ is the characteristic polynomial $\chi_A$ of $A$. We deduce from them that $|a_{d-k}|\leq \eps^k$ for each $k=1,\ldots,d$. If $\lambda$ is an eigenvalue of $A$, then $\chi_A(\lambda)=0$ and thus $$|\lambda|^d \leq \eps|\lambda|^{d-1}+\ldots+\eps^k|\lambda|^{d-k}+\ldots+\eps^d.$$ Setting $x=\eps/|\lambda|$, we obtain $1\leq x+\ldots+x^d$. But this implies $x\ge 1/2$.
\end{proof}

\begin{lemma}\label{convex} Let $n \in \N$ and $S\subset M_d(\C)$ be a bounded set such that $\eps:=\max_{k \leq nd} \Lambda(S^k)^{\frac{1}{k}} \leq 1$. Let $Q$ be the complex convex hull of $S\cup \ldots \cup S^n$. Then $\Lambda(Q)\leq 2d\eps$.
\end{lemma}

\begin{proof} Note that $Conv(A)Conv(B) \subset Conv(AB)$ for any two sets $A,B \in M_d(\C)$. So if $a \in Q$, then $a^k$ belongs to the convex hull of $S^k \cup \ldots \cup S^{nk}$. In particular
$$|\trace(a^k)| \leq d\eps^k \leq (d\eps)^k$$
for each $k=1,\ldots,d$. The conclusion now follows from Lemma \ref{trace}.
\end{proof}

We now prove Theorem \ref{pol-bd}. Rescaling and triangularizing $S$ if necessary, we may assume without loss of generality that $\rho(S)=1$ and that $S$ is irreducible. As in the proof of Theorem \ref{bg-el}  take a Barabanov norm $\|\cdot\|$ for $S$. Pick a unit vector $x \in \C^d$ and find $s_1,\ldots,s_n,\ldots$ in $S$ such that $\|x_n\|=1$ for all $n$, where $x_n:=s_n\cdot \ldots \cdot s_1x$. For $T$ and $\eps>0$ as in Lemma \ref{siegel} we obtain integers $c_i$'s not all zero such that $|c_i|\leq T$ and $\|\sum_1^n c_i x_i\| \leq \eps$. Let $i_0$ be the smallest index $i$ with $c_i \neq 0$ and set $y=x_{i_0}$. Hence we may write:
$$\|c_{i_0}y + \sum_{i>i_0} c_i s_i \cdot \ldots \cdot s_{i_0+1} y \|\leq \eps.$$ In other words:
\begin{equation}\label{Aeq}\|\lambda y +Ay\|\leq \frac{\eps}{N},\end{equation}
where $A:=\frac{1}{N} \sum_{i>i_0} -c_i s_i \cdot \ldots \cdot s_{i_0+1}$, $\lambda=\frac{c_{i_0}}{N}$ and $N:=\sum_{i>i_0} |c_i|$. Note that $N\neq 0$, because $\eps<1$ and $\|x_i\|=1$ for all $i$. Note further that $\|A\|\leq 1$ because $\|s\|\leq 1$ for all $s \in S$. And that $|\lambda|\ge \frac{1}{N} \ge \frac{1}{Tn}$.

If we can apply Lemma \ref{reg} to $A$ and $\lambda$, we will get $$\Lambda(A) \ge |\lambda|- 4(\frac{\eps}{N})^{\frac{1}{d}} \ge \frac{1}{2N} \ge \frac{1}{2nT}.$$ provided  $4(\frac{\eps}{N})^{\frac{1}{d}} \leq 1/2N$. The conditions for Lemma \ref{reg} require that $|\lambda|\leq 2$, while those for Lemma \ref{siegel} require  $(1+T)^{n} > (1+2nT\eps^{-1})^d$. These conditions will be fulfilled if we set $T=32d^2$,  $n=2d^2$ and $\eps^{-1}=8^d(nT)^{d-1}$. We conclude that 
$$\Lambda(A) \ge \frac{1}{2^7d^4}.$$
However $A$ belongs to the convex hull of $S \cup \ldots \cup S^n$. Therefore Lemma \ref{convex} implies that $$\max_{k \leq nd} \Lambda(S^k)^{\frac{1}{k}}\ge \frac{1}{2^8d^5}.$$
This yields the first inequality in Theorem \ref{pol-bd}. The second follows by applying the first to $S^m$ for $m=\lceil 8+5\log_2 d\rceil$.

\begin{proof}[Proof of Theorem \ref{bocnewA}] This is very similar.  Suppose $\|S\|=1$ and let $\delta=\|S^{n_1}\|$. Pick a unit vector $x$ and $s_1,\ldots,s_{n_1}$ such that $\|s_{n_1}\cdot \ldots \cdot s_1  x\|=\delta$. Arguing as in the above proof of Theorem \ref{pol-bd} we get a $y$ with $\|y\|\ge \delta$ such that $(\ref{Aeq})$ holds. Lemma \ref{reg} gives $\Lambda(A) \ge \frac{1}{2n_1T}$ if $\eps$ is chosen so that $ 4 (\eps \delta^{-1}/n_1T)^{1/d} = 1/2n_1T$. Then setting $n_1=2d^2$, $n_1T=M \delta^{-1}$, we see that the condition for Lemma \ref{siegel} is fulfilled if $M\ge2^6d^4$. But $\rho(S) \ge \Lambda(A) \ge \delta/2M$, proving the claim.
\end{proof}

\section{Ultrametric complete valued fields}
 In this section we consider the analogue of the above for an algebraically closed complete and non-archimedean valued field $K$ and prove Theorem \ref{ultra-thm}.

Let $\mathcal{O}:=\{x \in K, |x|\leq 1\}$ be the ring of integers,  $\mathfrak{m}:=\{x \in K, |x|<1\}$ its maximal ideal and $\overline{k}=\mathcal{O}/\mathfrak{m}$ the residue field. Recall that the value group of $K$ is dense in $\R_{>0}$ since $K$ is algebraically closed. By an ultrametric norm on $K^d$, we mean a map $\|\cdot\|:K \to \R_{\ge 0}$ such that $\|\lambda x\|=|\lambda|\|x\|$, $\|x+y\|\leq \max\{\|x\|,\|y\|\}$ and $\|x\|=0$ if and only if $x=0$, for all $x,y \in K^d$, $\lambda \in K$.

An \emph{orthogonal basis} for an ultrametric norm is a basis $(e_i)_1^d$ of $K^d$ such that $\|x\|=\max\{ c_i |x_i|\}$ for some positive reals $c_i$, if $x=x_1e_1+\ldots+x_de_d$. We say that it is orthonormal if $c_i=1$ for all $i$. If $K$ is locally compact, or just spherically complete \cite[2.4.4]{BGR}, all ultrametric norms admit an orthogonal basis, but in general we only have:

\begin{lemma}\label{eq-no} Let $v$ and $w$ be two ultrametric norms on $K^d$ and $\alpha>1$ a real.  Then there is $g \in \GL_d(K)$ such that $w(x)\leq v(gx) \leq \alpha w(x)$ for all $x \in K^d$.
\end{lemma}

\begin{proof} This is well-known and follows from the existence \cite[2.6.2 Prop. 3]{BGR} of almost orthogonal bases for ultrametric norms on $K^d$ and the density in $\R^+$ of the value group of $K$.
\end{proof}

We begin by pointing out that the Rota-Strang observation, Lemma \ref{rotas} and its proof, remain valid in the ultrametric setting. Combined with Lemma \ref{eq-no}, this yields the right hand side of $(\ref{boc-ultra})$.  It turns out that the infimum in $(\ref{rs})$ is realized under some mild conditions (milder than in the complex case):

\begin{lemma}\label{pd} Suppose that the value group of $K$ is all of $\R_{>0}$ and $S \subset M_d(K)$ is a bounded set. If $\rho(S)=0$, then $S^d=0$, while if $\rho(S)>0$, then there is an ultrametric norm $\|\cdot\|$ on $K^d$ with $\|S\|=\rho(S)$. 
\end{lemma}

\begin{proof}The first assertion follows from the same argument as in Lemma \ref{pded}. If $\rho(S)>0$, we may rescale and assume that $\rho(S)=1$, because we can  pick $\lambda \in K$ with $|\lambda|=\rho(S)$. If $S$ is irreducible, then the proof of Lemma \ref{pded} works verbatim and yields the desired norm. In general, we may choose a basis of $K^d$ for which $S$ is in block triangular form with irreducible blocks and define the norm $\|x\|=\max_i \|x_i\|_i$, where $\|\cdot\|_i$ is a norm on the $i$-th block with $\rho(S_{ii})=\|S_{ii}\|_i$, provided $S_{ii} \neq 0$ and arbitrary otherwise. We may further conjugate $S$ by a block diagonal matrix $g$, where the $i$-th block is the scalar matrix $t^i$ for some $t\in K$ with $0 \neq |t|<1/\|S\|$. Then, because of the ultrametric property, $\|gSg^{-1}\|\leq 1$. Thus $\|g \cdot g^{-1}\|$ is the desired norm. 
\end{proof}

We now proceed to the proof of Theorem \ref{ultra-thm}. It follows the same idea as in the proof of Claim 1 from the introduction, but we will need to palliate the lack of compactness, and the fact that the value group may not be all of $\R_{>0}$ by the use of an ultrapower construction. The gist of the proof is in the following lemma:

\begin{lemma}Suppose that $\|\cdot\|$ is an ultrametric norm admitting an orthonormal basis. If $S\subset M_d(K)$ is such that $\|S\|=\rho(S)=1$, then $$\max_{k \leq \ell(d)} \Lambda(S^k)=1.$$
\end{lemma}

\begin{proof} Let $(e_i)_1^d$ be the orthonormal basis, i.e. $\|x\|=\max_1^d|x_i|$ if $x=x_1e_1+\ldots + x_de_d$. In this basis, $S \subset M_d(\mathcal{O})$. Consider the convex hull $Q$ of $S,\ldots,S^{\ell(d)}$, that is the $\mathcal{O}$-module they span.  If $\Lambda(S^k)<1$ for each $k=1,\ldots,\ell(d)$, the characteristic polynomial of a matrix in $S^k$ will be $t^d$ modulo $\mathfrak{m}$.  So the image of $Q$ modulo $\mathfrak{m}$ in $M_d(\overline{k})$ will consist of nilpotent matrices and it will be a subalgebra of $M_d(\overline{k})$ by definition of $\ell(d)$. By Wedderburn's theorem it will therefore be a nilpotent algebra and we conclude that $S^d \subset M_d(\mathfrak{m})$. In particular $\|S^d\|<1$, contradicting our assumption that $\rho(S)=1$. 
\end{proof}

\begin{proof}[Proof of Theorem \ref{ultra-thm}] Suppose first that the value group of $K$ is all of $\R_{>0}$ and that all ultrametric norms on $K^d$ admit an orthonormal basis. Then the theorem follows from the combination of the two previous lemmas by renormalizing $S$. So to handle the general case, it is enough to show that $K$ can be embedded in another such field with the above properties. Any ultralimit $\mathbf{K}=\ell_\infty(K)/\equiv$ of $K$ with respect to some non-principal ultrafilter $\mathcal{U}$ on $\N$ will do. Here $\ell_\infty(K)$ is the space of bounded sequences in $K$ and $(x_n)_n \equiv (y_n)_n$ if and only if $\lim_{\mathcal{U}}|x_n-y_n|=0$. Indeed, by the countable saturation property of ultraproducts (e.g. \cite[2.25]{chatzidakis}), an ultralimit $\mathbf{K}$ will again be complete and algebraically closed, its value group will be $\R_{>0}$ and, because of Lemma \ref{eq-no}, all norms will admit an orthonormal basis. This shows Theorem \ref{ultra-thm} in full.
\end{proof}

\begin{proof}[Proof of Theorem \ref{bocA-ultra}]This follows from Bochi's original argument \cite[Theorem A]{bochi} suitably adapted to the ultrametric setting. First, up to passing to a suitable field extension as in the proof of Theorem \ref{ultra-thm}, we may assume that all norms admit an orthonormal basis. Pick one so that $\|x\|_0=\max_i |x_i|$. Then observe the following: for every invertible diagonal matrix $a$, we have:
\begin{equation}\label{inte}\|aS^da^{-1}\|_0 \leq \|S\|_0 \cdot \|aSa^{-1}\|_0^{d-1}.\end{equation}
Indeed every matrix entry of an element of $aS^da^{-1}$ is a sum of monomials of the form $a_{i_1} s^{(1)}_{i_1i_2} \cdot \ldots \cdot s^{(d)}_{i_{d}i_{d+1}} a_{i_{d+1}}^{-1}$ for matrices $s^{(i)}\in S$. We may write it as  $a_{i_1} s^{(1)}_{i_1i_2} a_{i_2}^{-1} a_{i_2} \cdot \ldots \cdot a_{i_{d}}^{-1}a_{i_{d}} s^{(d)}_{i_{d}i_{d+1}} a_{i_{d+1}}^{-1}$, a product of $d$ factors each bounded by $\|aSa^{-1}\|_0$. However at least one of the $d$ factors is bounded by $\|S\|_0$, because for at least one $j \in [1,d]$, $|a_{i_j}^{-1}a_{i_{j+1}}|\leq 1$, proving $(\ref{inte})$. 
Now we claim that $(\ref{inte})$ holds for an arbitrary matrix $a \in \GL_d(K)$, no longer assumed diagonal. Indeed this follows from the fact that $\|\cdot\|_0$ is invariant under $\GL_d(\mathcal{O})$ and that any matrix in $\GL_d(K)$ can be written as a product $k_1ak_2$, with $k_1,k_2$ in $\GL_d(\mathcal{O})$ and $a$ diagonal, as can be easily checked using operations on rows and columns as in gaussian elimination. Finally, the theorem is proved taking the infimum in over all $a \in \GL_d(K)$ in view of  $(\ref{rs})$.
\end{proof}

Finally we record one last observation.

\begin{proposition} If $S\subset M_d(K)$ is bounded and irreducible, then it admits a Barabanov norm, i.e. an ultrametric norm $\|\cdot\|$ such that $\max_{s \in S} \|sx\|=\rho(S)\|x\|$ for all $x \in K^d$. 
\end{proposition}

\begin{proof} By the proof of Theorem \ref{ultra-thm}, we may embed $K$ into a complete algebraically closed valued field $\mathbf{K}$ whose value group is all of $\R_{>0}$. Pick $\lambda \in \mathbf{K}$ with $|\lambda|=\rho(S)$. Then  Lemma \ref{pd} implies that $\lambda \neq 0$ and that $\mathbf{S}:=S/\lambda \subset M_d(\mathbf{K})$ is product bounded and admits an extremal norm $\|\cdot\|$. We may define the Barabanov norm of $S$ by the same formula $(\ref{tintynorm})$ applied to $\mathbf{S}$. Irreducibility forces this semi-norm to be a genuine norm.
\end{proof}

\bigskip 
\noindent \emph{Acknowlegements:} I am grateful to Cagri Sert for many interesting discussions about the joint spectral radius and for his remarks on this article. I am also indebted to Ian Morris for bringing my attention to Bochi's first inequality and to related references.

\bibliographystyle{plain}
\bibliography{bibliography}

\end{document}